\documentclass{article}
\usepackage{amsmath, amsfonts, amssymb, amsthm}
\usepackage{stmaryrd}
\usepackage{stackrel}
\usepackage{tikz-cd}
\usepackage{xcolor}
\usepackage{ytableau}
\usepackage[hidelinks]{hyperref}
\usepackage[all]{xy}
\usepackage{enumitem}

\newtheorem{theorem}{Theorem}[section]
\newtheorem{proposition}[theorem]{Proposition}
\newtheorem{corollary}[theorem]{Corollary}
\newtheorem{lemma}[theorem]{Lemma}
\newtheorem{remark}[theorem]{Remark}

\def\sT{{\mathsf{T}}}
\def\sA{{\mathsf{A}}}
\def\sS{{\mathsf{S}}}

\def\bd{{\mathbf{d}}}

\def\CV{{\mathcal{V}}}
\def\CO{{\mathcal{O}}}

\def\CM{{\mathcal{M}}}

\def\bla{{\boldsymbol{\lambda}}}
\def\bmu{{\boldsymbol{\mu}}}

\def\sheafhom{\mathop{\mathcal{H}\! \mathit{om}}}
\def\sheafend{\mathop{\mathcal{E}\! \mathit{nd}}}

\def\sslash{/\!/}
\def\gg{>\!\!>}

\newcommand{\hide}[1]{}
\newcommand{\todo}[1]{}

\newcommand{\nocontentsline}[3]{}
\let\origcontentsline\addcontentsline
\newcommand\stoptoc{\let\addcontentsline\nocontentsline}
\newcommand\resumetoc{\let\addcontentsline\origcontentsline}

\title{Relative Quasimaps and Tilting Module of $U(\mathfrak{gl}_n)$}
\author{Che Shen}
\date{}
 
\begin{document}
\setlength{\parindent}{0em}
\setlength{\parskip}{1em}
\maketitle

\begin{abstract}
We study the equivariant cohomology of the moduli space of quasimaps from $\mathbb{P}^1$ with one marked point to the flag variety. 
This moduli space has an open subset isomorphic to the Laumon space. The equivariant cohomology of the Laumon space carries a natural action of $U(\mathfrak{gl}_n)$ constructed via geometric correspondences. We extend this construction to the entire quasimap moduli space and relate it to tilting modules of $U(\mathfrak{gl}_n)$.
\end{abstract}

\tableofcontents

\section{Introduction}
\stoptoc
\subsection{Quasimaps}
Given a GIT quotient $W \sslash G$, where $W$ is an affine variety and $G$ is a reductive group\footnote{In this paper, all algebraic varieties and stacks are defined over $\mathbb{C}$.},
Quasimaps \cite{ciocan2014stable} from a curve $C$ to $W \sslash G$ are defined to be maps from $C$ to the stack
$[W/G]$ that generically land in the semi-stable locus $W^{ss}$. 
In this paper, we consider the case where 
\begin{equation}\label{eqn:W_mod_G}
    \begin{split}
        W &= \bigoplus _{i=1}^{n-1} \text{Hom}(\mathbb{C}^{i}, \mathbb{C}^{i+1}) \\
    G &= \prod _{i=1}^{n-1}GL(i)
    \end{split}
\end{equation}
for a fixed integer $n \geq 2$. $G$ acts on $W$ by conjugation and $W \sslash G$
is isomorphic to the flag variety of $\mathbb{C}^n$.  

\subsection{Nonsingular quasimaps, or Laumon spaces}\label{subsec:laumon}
A special case of the moduli space of quasimaps that has been studied extensively in the literature is the 
Laumon space, which parametrizes flags of locally free sheaves on $\mathbb{P}^1$:
\begin{equation} 
    \label{eqn:flagofsheaves}
    \CV_1 \subset \dots \subset \CV _{n-1} \subset \CO _{\mathbb{P}^1}^{\oplus n}
\end{equation}
such that $\CV_i$ has rank $i$ and 
the fibers at $\infty \in \mathbb{P}^1$ match a fixed full flag of subspaces of $\mathbb{C}^n$.
This can be viewed as quasimaps from $\mathbb{P}^1$ to the flag variety of $\mathbb{C}^n$ that are {nonsingular} at $\infty$ with a fixed evaluation map. 
\footnote{``Nonsingular'' here
means the evaluation map at $\infty $ lands in $W ^{ss}$. Not to be confused with
singular/smoothness of an algebraic variety.}
We denote it by $QM_{ns}$. See Section \ref{subsec:defs} for the naming conventions about quasimaps in this paper.

Fix a basis $e_1,...,e_n$ of $\mathbb{C}^n$.
The torus $\sT = (\mathbb{C}^*)^{n+1}$ acts on $QM_{ns} $ by scaling the $\mathbb{C}^n$ as well as the domain $\mathbb{P}^1$.
So we can consider the equivariant cohomology $H^*_{\sT}(QM_{ns})$ (resp. equivariant $K$-theory $K_\sT(QM_{ns})$). It was shown in 
\cite{feigin2011gelfand,braverman2005finite} (see also \cite{nakajima2011handsaw}) that one can construct an action of the universal enveloping algebra $U(\mathfrak{gl}_n)$
(resp. the quantum group $U_q(\mathfrak{gl}_n)$) via geometric correspondences. 
Furthermore,
$H^*_{\sT}(QM_{ns})$ can be identified with the ``universal dual Verma module'' of $U(\mathfrak{gl}_n)$. More precisely, $H^*_{\sT}(QM_{ns})$ is a module over
\[ 
    H^*_{\sT}(pt) = \mathbb{C}[a_1,...,a_n, \epsilon ].
\]
Let $\lambda =(\lambda _1,...,\lambda _n)$ and $\epsilon _0$ be any complex numbers. We can specialize parameters by the map
\[ 
    H^*_{\sT}(pt) \to \mathbb{C}_\lambda 
\]
given by 
\begin{equation}
    \label{spec_wt}
    \begin{split}
        a_i \mapsto  \lambda _i &\epsilon ,\; i=1,...,n \\
        \epsilon &\mapsto \epsilon _0
    \end{split}
\end{equation}
Then
\[ 
    H^*_{\sT}(QM_{ns}) \otimes _{H^*_{\sT}(pt)}\mathbb{C}_\lambda 
\]
can be identified with the dual Verma module of $U(\mathfrak{gl}_n)$ with lowest weight $\lambda - \rho $, where 
\[ 
    \rho = (-1, -2, ..., -n)
\]
is half sum of positive roots up to overall shift. (In this paper, we consider lowest weight modules
rather than highest weight modules. So some conventions differ from the usual ones by a sign.)
A similar result can be proved for $K_\sT(QM_{ns})$, see \cite{shen2024affine}.

\subsection{Relative quasimaps}
The problem of compactifying moduli spaces of maps 
has been a central theme in enumerative geometry. 
The construction of \cite{ciocan2014stable} provides a natural compactification of
the Laumon spaces by allowing the domain $\mathbb{P}^1$ to bubble up into a 
chain of $\mathbb{P}^1$'s. 
The precise definition will be 
discussed in Section \ref{subsec:defs}. 
We call this the moduli space of quasimaps with relative condition at $\infty $, as in \cite{okounkov2015lectures}, and denote it by $QM_{rel}$. 


We will establish some basic properties of $QM_{rel}$ and describe the $\sT$-fixed locus and 
the Bialynicki-Birula decomposition of a subset of it.
These techniques were widely used in studying smooth algebraic varieties with a torus action, but certain cautions are required when dealing with Deligne-Mumford stacks, as is the case for $QM_{rel}$. Similar analysis were also carried out
for the moduli space of stable maps in \cite{oprea2006tautological} and the recent work \cite{knutson2025stable}. 
\todo{
It's interesting 
to compare our results to theirs, see Remark.}



\subsection{}

The goal of this paper is to do a similar construction as in Section \ref{subsec:laumon}
for $H^*_{\sT}(QM_{rel})$.  
It turns out that a $U(\mathfrak{gl}_n)$ action can be constructed in essentially the same way.
The geometry of $QM_{rel}$ suggests that $H^*_{\sT}(QM_{rel})$ should be closely
related to the tilting modules of $U(\mathfrak{gl}_n)$. 
(We use the definition of \cite{humphreys2021representations}: 
A tilting module is a module in category $\mathcal{O}$ that has both Verma and dual Verma filtrations.) To see this, note that 
For a relative quasimap, if we fix the map on the ``bubbles'' but let the map on the 
parametrized $\mathbb{P}^1$ vary, we get a space that is isomorphic to $QM_{ns}$. Exploiting this, we get a filtration on $QM_{rel}$ such
that each filtered piece looks like $QM_{ns}$. (In practice, this filtration is 
constructed for a substack of $QM_{rel}$, see below.) Taking cohomology, this leads to a filtration by dual Verma modules. The properness of $QM_{rel}$ implies that the module we get is self-dual.

Two findings are worth mentioning: First, the modules we get are not tilting modules in the usual sense because 
they are not in category $\mathcal{O}$ -- The Cartan elements act non-semisimply. But they live in a dual category $\mathcal{O}'$ studied in \cite{soergel1986equivalences},
which is known to be equivalent to category $\mathcal{O}$. And these modules are the images of tilting modules under this equivalence.

Second, it turns out that it's more natural to consider a submodule rather than the whole $H^*_{\sT}(QM_{rel})$. Recall that 
to specialize to a certain highest weight $\lambda =(\lambda _1,...,\lambda _n)$, we need to consider  
\[ 
    H^*_{\sT}(QM_{rel}) \otimes _{H^*_{\sT}(pt)}\mathbb{C}_\lambda 
\]
as in \eqref{spec_wt}. This can be further identified with (non-equivariant) cohomology of the fixed locus
\[ 
    H^*\left(QM_{rel}^{\mathbb{C}^*_\lambda }\right)
\]
where $\mathbb{C}^*_\lambda \subset \sT $ is the subtorus determined by $\lambda $. By examining the
connected components of the fixed locus, we see that this module decomposes into a direct sum. We study the direct summand coming from the component that intersects $QM_{ns}$. We can determine its graded dimension and thus determine how it decomposes into indecomposible tilting modules. 
However, it is still an open question to geometrically characterize an 
indecomposible tilting module in this setting.

\subsection{}
Our approach also 
applies to $K_\sT(QM_{rel})$, but some subtleties needs to be taken into account, see Remark \ref{rmk:uq_action}.
It should also be mentioned that \cite{feigin2011yangians}, \cite{tsymbaliuk2010quantum}, \cite{neguct2018affine}
constructed a larger algebra -- the Yangian $Y(\mathfrak{gl}_n)$ or the quantum affine algebra $U_q(\widehat{\mathfrak{gl}}_n)$ --
acting on $H^*_{\sT}(QM_{ns})$ or $K_\sT(QM_{ns})$. Our results can also be generalized in 
this direction, see Remark \ref{rmk:bigger_alg}.
However, we will focus on the $U(\mathfrak{gl}_n)$ action in this paper.

\subsection{Outline of the paper}
We introduce the notations and study the geometry 
of $QM_{rel}$ in detail in Section \ref{sec:quasimap} and establish several technical results to be used later. 
The construction of $U(\mathfrak{gl}_n)$ action on 
$H^*_{\sT}(QM_{ns})$ is recalled in Section \ref{sec:action}, where we also extend this action to $H^*_{\sT}(QM_{rel})$. In Section \ref{sec:spec_lowest_wt},
we specialize the equivariant parameters and single out a direct summand $H_{\lambda , w}$ by analyzing the fixed locus. 
In Section \ref{sec:Hpw}, we show that $H_{\lambda,w}$ 
is a tilting module in category $\mathcal{O}'$ 
and compute its graded dimensions.

\resumetoc

\subsection*{Acknowledgement}
I would like to thank Andrei Okounkov for suggesting this project and for his guidance throughout its development.
I am grateful to Tommaso Botta, Yixuan Li, Melissa Liu, Peng Shan, Andrey Smirnov for many helpful and enlightening discussions. I also want to thank Younghan Bae, Dragos Oprea and Wolfgang Soergel for answering my questions and pointing me to valuable references.

\section{The Moduli Space of Quasimaps}\label{sec:quasimap}

\subsection{Definitions and conventions}\label{subsec:defs}
To begin with, we recall the definition of quasimaps. Concretely,
a quasimap from a curve $C$ to a GIT quotient $W \sslash G$ is 
a principal $G$ bundle $\mathcal{G}$ together with a section of the 
associated $W$ bundle
\[ 
    f \in \Gamma(C, \mathcal{G}\times _{G}W)
\]
with certain stability conditions. 


In this paper, we will study two variants of the moduli space of quasimaps.
We refer the readers to \cite{okounkov2015lectures} for the general 
definitions, but spell out the definitions when the target is the flag variety. Note that the naming conventions here differ from the ones used in \cite{ciocan2014stable}. What we call $QM_{rel}$ here should be called quasimap from a parametrized $\mathbb{P}^1$ with one marked point at $\infty $ to the flag variety in \cite{ciocan2014stable} (see Section 7.2 therein), while $QM_{ns}$ is 
the open subset in it where the domain curve does not degenerate.

Fix an integer $n \geq 2$.
Let $\text{Flag}(\mathbb{C}^{n})$ be the full flag variety of $\mathbb{C}^{n}$ and 
let $x$ be a point in it.
We use $QM_{ns, x}$ to denote the moduli space of quasimaps to the flag variety (written as a GIT quotient $W \sslash G$ as in \eqref{eqn:W_mod_G}) with nonsingular condition at $\infty $ and $ev(\infty )=x$. 
Unwinding the definition, this is the data of flags of sheaves
\[ 
    \CV_1 \subset \dots \subset \CV _{n-1} \subset \CO _{\mathbb{P}^1}^{\oplus n}
\]
such that $\mathcal{V}_i$ has rank $i$ and the restriction to $\infty $
\[ 
    \CV_1|_\infty  \subset \dots \subset \CV _{n-1}|_\infty  \subset \mathbb{C}^{n}
\]
corresponds to the point $x$ in $\text{Flag}(\mathbb{C}^{n})$. 

Let $QM_{rel,x}$ denotes quasimaps to $\text{Flag}(\mathbb{C}^{n})$ with \textit{relative} condition at $\infty $ and $ev(\infty )=x$. A quasimap in $QM_{rel,x}$ is the data of 
\begin{center}
\begin{tikzcd}
C \arrow[r, "{(\mathcal{G}, f)}"] \arrow[d, "\varphi"] & {[W/G]} \\
\mathbb{P}^1                                      &        
\end{tikzcd}
\end{center}
\begin{itemize}
    \item A domain curve $C$ which is either a $\mathbb{P}^1$ or a nodal curve consisting of a chain of $\mathbb{P}^1$'s, together with a marked point (denoted $\infty $) on the rightmost $\mathbb{P}^1$.

    \item A map $\varphi: C \to \mathbb{P}^1$ that restricts to an isomorphism on the leftmost $\mathbb{P}^1$, while the complement in $C$ maps to a point. We will refer to the leftmost $\mathbb{P}^1$ as the parametrized $\mathbb{P}^1$ and other $\mathbb{P}^1$'s as bubbles.
    
    \item A principal $G$ bundle $\mathcal{G}$ together with a section
    \[ 
        f \in \Gamma(C, \mathcal{G}\times _{G}W)
    \]
    such that $f(\infty )=x$ and $f$
    satisfies the stability condition that 
    \begin{itemize}
        \item all the nodes and the marked point $\infty $ maps to $W\sslash G$ under $f$. (This implies that there could be only finitely many points that maps to the unstable locus in $W$.)
        \item For any $\mathbb{P}^1$ in the bubble, $f$ is not a constant map when restricted to that $\mathbb{P}^1$.
    \end{itemize}

\end{itemize} 

Two quasimaps $(C_1, \mathcal{G}_1, f_1)$ and $(C_2, \mathcal{G}_2, f_2)$ are isomorphic if there is an isomorphism $g: C_1 \xrightarrow{\sim}C_2$ such that $g(\infty )=\infty $ and the diagram 
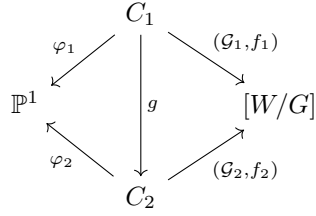
\begin{center}
\begin{tikzcd}
     & C_1 \arrow[ld, "\varphi_1"'] \arrow[rd, "{(\mathcal{G}_1, f_1)}"] \arrow[dd, "g"] &         \\
\mathbb{P}^1 &                                                               & {[W/G]} \\
     & C_2 \arrow[lu, "\varphi_2"] \arrow[ru, "{(\mathcal{G}_2, f_2)}"']                 &        
\end{tikzcd}
\end{center}
commutes. This means that if two quasimaps differ by an automorphism of a $\mathbb{P}^1$ in the bubbles, they are considered
isomorphic. But automorphism of the parametrized $\mathbb{P}^1$ can give a different quasimap. In fact, we will consider the $\mathbb{C}^*$ action on $QM_{rel}$ and $QM_{ns}$ induced by the $\mathbb{C}^*$ action on the parametrized $\mathbb{P}^1$.

As we have seen in the case of $QM_{ns}$, the data of $\mathcal{G}$ and $f$ is equivalent to vector bundles $\mathcal{V}_i, i=1,...,n-1$ on $C$ together with morphisms $\mathcal{V}_{i-1}\to \mathcal{V}_i$.
The degree of a quasimap is defined to be the tuple
\[ 
    \bd =(d_1,...,d_{n-1})
\]
where 
\[ 
    d_i = \deg \mathcal{V}_i,\; i=1,...,n-1
\]
If $f$ is an actual map to $\text{Flag}(\mathbb{C}^n)$, then this is equivalent to defining the degree using the homology class of the image of $f$.
The connected components in $QM_{ns}$ and $QM_{rel}$ are parametrized by the degree. We may write $QM_{ns}^{\bd }$ and $QM_{rel}^{\bd }$ to denote degree $\bd $ quasimaps.

We will also need the space $\widetilde{QM}_{ev(\infty )=x}$ in the proofs. This is the moduli space of quasimaps with no parametrized component and $ev(\infty )=x$. In other words, $\widetilde{QM}$ is the bubble part of $QM_{rel}$.
Since the target is almost always $\text{Flag}(\mathbb{C}^{n})$ in this paper, we omit it from the notation. The only place where this is not the case is in Section \ref{subsec:smoothness}, where the target can be the Grassmannian, and we write it as $QM_{rel}(Gr(k,n))$.

\textbf{Torus action.}
Fix a basis $e_1,...,e_n$ of $\mathbb{C}^{n}$. The torus $\sA = (\mathbb{C}^*)^{n}$ acts on it by scaling the coordinates. Let $x_0$ be the standard flag. 
The $\sA$-fixed points of the flag variety have the form $w(x_0)$ for $w \in W$
the symmetric group of $n$ elements. When talking about the $\sA$ action, $ev(\infty )$ must be one of these points. Choosing a different point corresponds 
to permuting the equivariant variables. We sometimes drop the $w(x_0)$ and simply write $QM_{rel}$, $QM_{ns}$, etc. to declutter the notations.

Let $\sT = \sA \times \mathbb{C}^*_\epsilon $ denotes the torus acting on $QM$, where $\mathbb{C}^*_\epsilon $ acts on the domain.
$a_1,...,a_n, \epsilon $ denote equivariant variables. 
For any group $G$, we define $R(G):=H^*_{G}(pt)$. For example, $R(\sT) = \mathbb{C}[a_1,...,a_n, \epsilon ]$. 
We will talk about the weights of a torus action using cohomological notation (e.g. $a_1-a_2 + \epsilon $).

\textbf{Universal curve and tautological classes.}
Let $QM$ denote $QM_{ns}$ or $QM_{rel}$ or {\footnotesize ${\widetilde{QM}}$}. Let $\pi : \mathfrak{C} \to QM$ 
denote the universal curve on it. So each fiber of $\pi $ is either a $\mathbb{P}^1$ or a chain of $\mathbb{P}^1$'s. It has two sections denoted by $0$ and $\infty $ which maps to the corresponding points in the fiber. For $QM_{rel}$, when the domain bubbles up,
we will use $\infty '$ to denote the node on the parametrized component. (This does not give a section of $\pi $.)

We use $\mathcal{V}_i$ to denote the tautological bundles on $\mathfrak{C}$ for $i=1,...,n-1$ as in \eqref{eqn:flagofsheaves}. Let $\mathcal{F}_i:=0^* \mathcal{V}_i$.

Throughout the paper we use $\mathbb{C}$ coefficients for cohomology and for Chow groups. So 
$H^*(X)$ stands for $H^*(X, \mathbb{C})$ and $A_{\mathbb{C}}^*(X)$ stands for $A^*(X)\otimes \mathbb{C}$.
We will see that $A^*$ and $H^*$ are isomorphic for the spaces we consider.

\subsection{Smoothness of $QM_{rel}$} \label{subsec:smoothness}
Many results below relies on the smoothness of $QM_{rel}$, so we establish it here.

\begin{proposition}
$QM_{rel}$ is a smooth Deligne-Mumford stack.
\end{proposition}
\begin{proof}
As discussed in Section 5 of \cite{ciocan2014stable}, $QM_{rel}$ has a perfect obstruction theory (i.e. a two term complex that maps to the cotangent complex that induces isomorphism on $H^0$ and surjection on $H ^{-1}$) defined as follows:
Let $\mathfrak{M}$ be the moduli stack of the domain curve. Then the relative obstruction theory over $\mathfrak{M}$ is given by 
\[ 
    \left(R ^\bullet \pi _*(\mathcal{Q})\right)^\vee 
\]
where $\mathcal{Q}$ is defined by the sequence
\begin{equation}\label{eqn:tvir}
    0 \to \bigoplus_{i=1}^{n-1} \sheafend (\mathcal{V}_i) \xrightarrow[]{\phi } \bigoplus_{i=1}^{n-1}\sheafhom(\mathcal{V}_{i}, \mathcal{V}_{i+1}) \to \mathcal{Q} \to 0
\end{equation}
Here, the map $\phi $ is defined in the following way: given a quasimap 
\[ 
    u = (u_1,...,u _{n-1}) \in \bigoplus_{i=1}^{n-1} \text{Hom}(\mathcal{V}_i, \mathcal{V}_{i+1}),
\]
(Note that the RHS is Hom, not sheaf $\sheafhom$.)
an element
\[ 
    \alpha =(\alpha _1,...,\alpha _{n-1}) \in \bigoplus_{i=1}^{n-1} \sheafend (\mathcal{V}_i)
\]
is sent to an element 
\[ 
    \beta =(\beta _1,...,\beta _{n-1}) \in \bigoplus_{i=1}^{n-1}\sheafhom(\mathcal{V}_{i}, \mathcal{V}_{i+1})
\]
by
\[ 
    \beta _i = u _{i} \circ \alpha _i + \alpha _{i+1}\circ u_i
\]
(and we let $\alpha _0 =  \alpha _n = 0$.)
\hide{Is it enough to check stalks on $\mathbb{C}$ points? Should be fine since one can pull back to a scheme first then pullback to a point.} 
To show that $QM_{rel}$ is smooth, it suffices to show that the $-1$ term of the perfect obstruction theory vanishes.
\hide{Remark that for fixed $ev(\infty)$ we need to twist the two terms by $\infty$, but the argument remains the same. Or use that not fixing $\infty$ is a locally trivial fibration over target.}

Since the perfect obstruction theory is equivariant with respect to the action of $\sT$, it suffices to show that the -1 term vanishes at $\sT$-fixed points. 
Fix a quasimap $u = (u_1,...,u _{n-1}) \in  QM_{rel}^\sT$. Let $C$ be the domain curve of $u$ and let $C'$ be the open subset of the curve $C$ removing $0, \infty$ and all the nodes. The cohomology of the stalk of $R ^\bullet \pi _*(\mathcal{Q})$ at $f$, denoted by $T^0$ and $T ^{1}$, fits into the long exact sequence induced by \eqref{eqn:tvir}:
\begin{align*}
    0 \to H^0 \left(\bigoplus_{i=1}^{n-1} \sheafend (\mathcal{V}_i)\right) \to  H^0\left(\bigoplus_{i=1}^{n-1}\sheafhom(\mathcal{V}_{i}, \mathcal{V}_{i+1})\right) \to T^0 & \\ 
    \to H^1 \left(\bigoplus_{i=1}^{n-1} \sheafend (\mathcal{V}_i)\right) \xrightarrow[]{\underline{\phi }}  H^1\left(\bigoplus_{i=1}^{n-1}\sheafhom(\mathcal{V}_{i}, \mathcal{V}_{i+1})\right) \to T^1 &
    \to 0 \\
\end{align*} 
We want to show that $T^1 = 0$ for any $u$, which is equivalent to show that $\underline{\phi }$ is surjective. To this end, note that any element in  $ H^1\left(\bigoplus_{i=1}^{n-1}\sheafhom(\mathcal{V}_{i}, \mathcal{V}_{i+1})\right) $ is represented by a section over $C'$
\[ 
\underline{\beta } = (\underline{\beta }_1,..., \underline{\beta }_{n-1})     
\in \Gamma (C', \bigoplus_{i=1}^{n-1}\sheafhom(\mathcal{V}_{i}, \mathcal{V}_{i+1}))
\]
that does not extend over any point in $C \backslash C'$. We define
\[ 
\underline{\alpha } = (\underline{\alpha }_1,..., \underline{\alpha }_{n-1})     
 \in \Gamma (C', \bigoplus_{i=1}^{n-1}\sheafend(\mathcal{V}_{i}))
\]
inductively: $\underline{\alpha }_1$ is a scalar and can be chosen arbitrarily. Once $\underline{\alpha }_i$ is chosen, $\underline{\alpha} _{i+1}$ is chosen so that 
\[ 
\underline{\alpha} _{i+1}\circ u_i=
    \underline{\beta} _i - u _{i} \circ 
    \underline{\alpha} _i 
\]
Since $u$ is fixed by $\sT$, each $u_i$ must be injective \textit{pointwise} on $C'$. 
Thus, the section $\underline{\alpha }_{i+1}$ above  always exists. By construction, $\underline{\phi }(\underline{\alpha }) = \underline{\beta }$.
\hide{Why other points don't lead to a contradiction if it has a pole somewhere? I think it's because one has the freedom to choose representatives. The $\alpha $ we are looking for may not map to $\beta $ on the nose, but only up to something extendable. This helps to move poles around.}
This shows that $\underline{\phi }$ is surjective, so $T^1 = 0$.  
\end{proof}

\hide{So far we have considered the relative obstruction theory with respect to the forgetful map $QM_{rel}\to \mathfrak{M}$. One can use Proposition 3 of \cite{kim2003functoriality} to turn this into an absolute obstruction theory. The cotangent complex of $\mathfrak{M}$ has no degree -1 cohomology. Thus, the absolute obstruction theory has no degree -1 cohomology either.}
\hide{What's the best way to compute cotangent complex of $\mathfrak{M}$? I think using https://mathoverflow.net/questions/169025/cotangent-complexes-of-quotient-stacks we can show that the $\mathfrak{g}$ map is injective so has no $H ^{-1}$.}
\hide{Another option is to use smoothness of $\mathfrak{M}$.}

\subsection{\texorpdfstring{$QM_{rel}$}{QMrel} as a global quotient}
\todo{Change this to a Proposition}
\begin{proposition}\label{prop:global}
$QM_{rel}$ is isomorphic to a quotient stack $[X/G]$ where $X$ is a quasi-projective scheme and $G = (\mathbb{C}^*)^{N}$ for some $N$. In addition, there is a torus $\sT$ action on $X$ that descends to the action of $\sT$ on $QM_{rel}$.
\end{proposition}

\begin{proof}
The construction can be carried out similar to \cite{marian2011moduli} Section 6 using Quot schemes. The key differences are: (1) We are considering quasimap from \textit{parametrized} $\mathbb{P}^1$ and (2) The target is the flag variety instead of Grassmannian.

As discussed in Section 6.4.4 of \cite{okounkov2015lectures}, for any given $N$, there is a universal curve $\mathfrak{C}_N$ living over $\mathbb{C}^n$ such that for any subset $I \subset \{1,2,...,N\}$, the fiber over a point in
\begin{equation}\label{eqn:openUI}
    U_I:= \{(x_1,...,x_N)\in \mathbb{C}^N| x_i=0 \text{ if and only if }i \in I\}
\end{equation}
is a chain of $|I|+1$ $\mathbb{P}^1$'s. Each divisor $\{x_i=0\}$ is the loci where the $i$-th node remains intact.

Fix a degree $\bd =(d_1,...,d _{n-1})$. Let $N = |\bd | = d_1 + ... + d _{n-1}$.
Consider the product
\[ 
    Q_0 := Q(n-1,d_1) \times _{\mathfrak{C}_N} Q(n-2,d_2) \times _{\mathfrak{C}_N} ... \times _{\mathfrak{C}_N} Q(1, d _{n-1})
\]
where $Q(n-r,d)$ stands for the relative Quot scheme parametrizing quotients
\[ 
    \mathcal{O}_C^n \to Q_r \to 0
\] 
where $Q$ has rank $n-r$ and degree $d$. 

Let ${V}_r$ be the kernel of the map $ \mathcal{O}_C^n \to Q_r$ for each $r$. Then $V_r$ is locally free of rank $r$ for $r=1,2,...,n-1$. Now let $Q_1$ be the open subset of $Q_0$ such that $V_r \to \mathcal{O}_C^n$ is injective on the nodes and the marked point, and that each $\mathbb{P}^1$ component has degree at least 1. Let $Q_2$ be the closed subset of $Q_1$ where $V_r$ is a subsheaf of $V _{r+1}$ for each $r=1,2,...,n-1$.

Now we are close to what we want. But note that the curves over $U_I$ defined in \eqref{eqn:openUI} for each $I$ may correspond to isomorphic curves when two sets $I_1$ and $I_2$ have the same cardinality. To avoid parametrizing the same quasimap twice, the quasimaps over them should have different degrees on each component $\mathbb{P}^1$. More precisely, let $\pi : Q_2 \to \mathbb{C}^{N}$ denote the projection.
For each 
\[ 
    I = \{i_1<i_2<...<i _{|I|}\},
\]
consider the subset of $\pi ^{-1}(U_I)$ where the quasimap on the $k$-th $\mathbb{P}^1$ component in the domain has total degree (i.e. sum of degrees of $\CV_1,...,\CV _{n-1}$ on this $\mathbb{P}^1$) equal to $i _{k}-i_{k-1}$, where $k=1,2,...,|I|+1$ and we set $i _{|I|+1} = N+1, i_0=1$. When $I$ ranges over all possible subsets, this gives a subset $Q_3 \subset Q_2$. By analyzing how node smoothing changes the degree, we see that $Q_3$ is an open subset of $Q_2$. 

The action of $(\mathbb{C}^*)^{N}$ lifts to $Q_3$. And there is a natural $\sT$ action on $Q_3$ commuting with the action of $(\mathbb{C}^*)^{N}$. With these constructions, the moduli space $QM_{rel}^\bd $ is isomorphic to the stack $[Q_3/(\mathbb{C}^*)^{N}]$.
\end{proof}

\begin{remark}
At fixed points that are orbifold points, the tangent weight usually involves fractional multiples of equivariant variables. However, in this quotient construction, 
the group $\sT$ acts on $Q_3$ without taking multiple covers. 
This is not a contradiction. A prototypical example is to have 
$U = \mathbb{C} \times \mathbb{C}^*$ with an action of $\mathbb{C}^*_t$ with weight 
$(t,t^2)$ and $\mathbb{C}^*_a$ with weight $(1,a)$. 
Then $[U/\mathbb{C}^*_t]$ is isomorphic to 
$[\mathbb{C}/\mu_2]$, 
and the action of 
$\mathbb{C}^*_a$ on $\mathbb{C}$ is $\sqrt{a}$. 
This happens when we consider $QM_{rel}^{d=2}(Gr(1,2))$.
\end{remark}

\subsection{Fixed points and tangent spaces} \label{subsec:fixed_locus}

\subsubsection*{Fixed Points in $QM_{ns, w(x_0)}$}
As discussed in \cite{braverman2005finite}, \cite{feigin2011gelfand}, fixed points in $QM_{ns, w(x_0)}$ are parametrized by tuples of non-negative integers
\[ 
    d _{i, w(j)},\; i=1,...,n-1,\; j = 1,...,i
\]
such that $d _{i_1, w(j)} > d _{i_2, w(j)}$ if $i_1<i_2$.
This corresponds to a quasimap where the vector bundles $\mathcal{V}_k, k=1,...,n-1$ decomposes as 
\[ 
    \mathcal{V}_k = \sum_{i=1}^k a _{w(i)} \mathcal{O}_{\mathbb{P}^1}(-d _{k, w(i)})
\]
with obvious maps between the $\mathcal{V}_k$'s.

\subsubsection*{Fixed Points in $QM_{rel, w(x_0)}$}
\begin{proposition}\label{prop:fp}
Fix $w \in W$.
Let $f \in QM_{rel,w(x_0)}$ be a $\sT$-fixed point and assume the domain of $f$ is a 
chain of $N+1$ $\mathbb{P}^1$'s. Let $r_1,...,r_N$ denote the nodes of the domain. Then
\begin{enumerate} [label=(\arabic*)]
    \item There exists $w_1,...,w_N \in W$ such that $f(r_i) = w_i(x_0)$. \label{enum:fp1}
    \item Let $f' = f| _{\text{parametrized }\mathbb{P}^1}$. Then $f'$ is a $\sT$-fixed point in $QM_{ns, w_1(x_0)}$.
    \item Let $w _{N+1}=w$. Then for each $i=1,...,N$, either $w _{i+1}=w_i$ or $w _{i+1} = s w_i$ for some simple reflection $s \in W$. \label{enum:fp3}
    \item If $w_i \neq w _{i+1}$ for some $i$, then the $(i+1)-$th $\mathbb{P}^1$ in 
    the domain maps to the $\mathbb{P}^1$ in the flag variety connecting $w_i(x_0)$ and $w _{i+1}(x_0)$. 
    This map is a $d_i$-fold covering for some $d_i \in \mathbb{Z}_+$. \label{enum:fp4}
    
    \item If $w_i = w _{i+1}$ for some $i$, then the map on the $(i+1)-$th $\mathbb{P}^1$ in 
    the domain is a ``constant'' map (with singular points). The fixed points may be non-isolated in this case.
    \item If $w_i \neq w _{i+1}$ for all $i=1,..,N$, then $f', w_i, d_i, i=1,...,N$ determine $f$ uniquely, and $f$ is an isolated fixed point. \label{enum:fp6}
\end{enumerate}
\end{proposition}

\begin{proof}
For part \ref{enum:fp3} and \ref{enum:fp4}, since each $\mathbb{P}^1$ in the bubble,
the map is invariant under $\sT$ if and only if the action can be compensated by scaling the $\mathbb{P}^1$.
This means that if $w _{i} \neq w _{i+1}$, then there is no singular point on the $(i+1)$-th $\mathbb{P}^1$. 
(Otherwise the scaling will move it.) So it is an actual map from $\mathbb{P}^1$
to the flag variety. The image must be an invariant $\mathbb{P}^1$ in the flag variety, hence the conclusion. Other parts of the statement are easy.
\end{proof}

\begin{remark}\label{rmk:p1weight}
In the situation of part \ref{enum:fp6} above, we will say the weight of the $(i+1)-$th $\mathbb{P}^1$ is the weight of its image in $\textnormal{Flag}(\mathbb{C}^{n})$ divided by $d_i$.   
\end{remark}

We will use notations like $(\bla, \widetilde{f})$ or $(\bla, \widetilde{P})$ to 
denote a fixed component where $\bla$ is a fixed point in $QM_{ns}$ and $\widetilde{f}$
or $\widetilde{P}$ is a fixed component in $\widetilde{QM}$.
The tangent space at fixed components split into contribution from $QM_{ns}$ and $\widetilde{QM}$:

\begin{proposition}\label{prop:tangent_add}
Let $(\bla , \widetilde{P})$ be a fixed component in $QM_{rel}$ where $\bla$ is in $QM_{ns, w(x_0)}$ for some $w \in W$. Then 
\[ 
    T _{(\bla, \widetilde{P})}QM_{rel} = T_\bla QM_{ns, w(x_0)} \boxplus (T _{\widetilde{P}} \widetilde{QM} \oplus \psi')
\]
where $\psi'$ on $\widetilde{QM}$ is the tangent line of the domain curve at $0$ (with a weight $(-1)$ $\mathbb{C}^*_\epsilon $ action.)
\end{proposition}
\begin{proof}
Note that
\[ 
    T QM_{rel} = T _{\text{fixed domain}} + \text{Deformation of nodes} - \text{Automorphism of bubbles}
\]
For each node $r$, deformation of the node is given by tensoring the tangent line 
of the two $\mathbb{P}^1$'s adjacent to it. The tangent line from $QM_{ns}$ is trivial
with a weight $(-1)$ $\mathbb{C}^*_\epsilon $ action because it's always a $\mathbb{P}^1$, while the tangent line from $\widetilde{QM}$ is $\psi '$.
\end{proof}

The following lemma follows immediately from the quotient construction and the description of fixed points.

\begin{lemma}
In the setting of Theorem \ref{prop:global}, let $\pi : X \to [X/G]$ denote the projection map and let $F$ be a connected component in the $\sT$-fixed locus of $[X/G]$. Then, possibly after replacing $\sT$ by a multiple cover $\widetilde{\sT}$, the action of $\widetilde{\sT}$ on $X$ can be chosen so that the action on $\pi ^{-1}(F)$ is trivial. 
\end{lemma}


\subsubsection*{Equivariant Localization}
Equivariant localization in cohomology \cite{atiyah1984moment} also applies 
to orbifolds. The analogous result for equivariant Chow group 
is discussed in \cite{edidin1996equivariant} and the appendix of \cite{graber1997localization}. 
In our setting, this says

\begin{theorem}
For any $\alpha  \in H^*_{\sT}(QM_{rel})$, 
\[ 
    \alpha = \sum_{P \in \textnormal{connected components of }QM_{rel}^\sT} \frac{i_{P*}i_P^*\alpha }{e(N_P)}
\]
where $i_P: P \to QM_{rel}$ is the inclusion and $N_P$ is the tangent bundle of $P$.
This is an equality in $H^*_{\sT}(QM_{rel})_{loc}:= H^*_{\sT}(QM_{rel})\otimes _{R(T)}\mathrm{Frac}R(T)$
\end{theorem}

For any fixed component $P$ and any element $\alpha \in H^*(P)$, let 
\[ 
    |\alpha  \rangle _{P} := \frac{i_{P, *} \alpha}{e(N_P)} \in H^*_{\sT}(QM_{rel})_{loc}
\]
where $i _{P,*}: P \to QM_{rel}$ is the inclusion. (The division by normal bundle is for convenience.
Under this definition, we have $i_P^* |\alpha  \rangle_P  = \alpha $)

\subsubsection*{Tautological Bundles}
We describe the weights of the tautological bundles $\mathcal{F}_i$ 
defined in Section \ref{subsec:defs} at the fixed components.

Recall that for any fixed component $(\bla, \widetilde{P})$ in $QM_{rel}$, the $\bla$ part is parametrized by a permutation $\sigma \in W$ determined by $ev(\infty ')$ 
and integers $d _{k, \sigma (i)}, k=1,...,n-1, i=1,...,k$ determined by the bundles $\mathcal{V}_i$ restricted to the parametrized $\mathbb{P}^1$. Let 
\[ 
    d_k = \deg \mathcal{V}_k \big|_{\text{parametrized }\mathbb{P}^1} = \sum_{i=1}^{k}d _{k, \sigma (i)}
\] 

\begin{proposition}
For any $k$, the bundle $\mathcal{F}_k| _{(\bla, \widetilde{P})}$ is trivial. Let $\sigma$ and $d _{k}$ be defined as above, then
\[ 
    \mathcal{F}_k| _{(\bla, \widetilde{P})} = -d_k \epsilon  + \sum_{i=1}^k a _{\sigma (i)}
\]
\end{proposition}

\subsection{Bialynicki-Birula decomposition}


Let $\mathbb{C}^* \subset \sT$ be a generic subtorus so that the $\mathbb{C}^*$ fixed points are the same as $\sT$ fixed points.
The Bialynicki-Birula decomposition was initially proved for proper smooth algebraic varieties in \cite{bialynicki1973some}, and generalized to Deligne-Mumford stacks in \cite{alper2020luna}.

\begin{theorem}(\cite{alper2020luna}, Theorem 5.27) 
Let $\{F_i\}_{i \in I}$ denote the connected components of the $\sT$-fixed locus of $QM_{rel}$ and use $F_i^+$ to denote the attracting set of $F_i$ under the $\mathbb{C}^*$ action. Then each $F_i^+$ is an affine fibration over $F_i$ and $QM_{rel}$ is the disjoint union of $F_i^+$.
\end{theorem}


As noted in Remark 5.29 of \cite{alper2020luna}, there
are some subtleties about whether B-B decomposition 
induces a stratification in the case of Deligne-Mumford stacks. In our case, writing the moduli space as a global quotient resolves this issue.

\begin{proposition}
The B-B decomposition of $QM_{rel}$ is a stratification, i.e. there is a partial ordering $\leq $ on the connected components $F_i$ such that 
\begin{equation}
    \label{eqn:closure_order}
    \overline{F_i^+} \subset \bigcup _{j \leq i}F_j^+
\end{equation}
\end{proposition}

\begin{proof}
We define the partial ordering using an ample line bundle, c.f. Section 3.2.4 of \cite{maulik2012quantum}.
Write $QM_{rel}$ as $[X/G]$ with a $\sT$ action on $X$. The variety $X$ is quasi-projective, so
by \cite{chriss1997representation} Corollary 5.1.21, there exists a $G \times \sT$-equivariant ample line bundle on $X$. It descends to a $G$-equivariant ample line bundle on $[X/G]$, which we denote by $L$. 
\hide{(I don't know what ample means exactly on DM stacks. What we need here is that on each $\mathbb{P}^1$ the weight decreases if we go along the attracting direction. On weighted $\mathbb{P}^1$, this is implied by the base-point free property.)}
Consider the torus $\mathbb{C}^*$ used to define the B-B decomposition. Define
\[ 
    F_i < F_j \text{ if }\text{weight of }L|_{F_i}<\text{weight of }L|_{F_j}
\]
Under this ordering, if $F_j$ is in the closure $\overline{F_i^+}$, then either $F_i$ and $F_j$ are connected by a $\mathbb{P}^1$, or there exists $F_k$ such that $F_k>F_j$ and $F_k$ is in the closure of $\overline{F_i^+}$. So one can prove inductively that the property \eqref{eqn:closure_order} is satisfied. \hide{One takes a repelling direction of $F_j$ and  sees what it connects to. If it connects to $F_k$ with $k \neq i$, then we can show that this path is also in the closure. This comes from taking a sequence that limits to $F_j$ then scaling them to make their coordinates equal on the repelling direction. It's a little tricky.}
\end{proof}

Using this stratification and exploiting the Gysin exact sequence in cohomology, see e.g. \cite{oprea2006tautological} Section 2, \cite{chriss1997representation} Section 5.5, we have 

\begin{proposition}
\label{prop:eqv_formality}

\begin{enumerate}[label=(\arabic*)]
    \item The natural map 
    \[ 
        A^*_{\sT}(QM_{rel})\to H_\sT^*(QM_{rel})
    \] 
    is an isomorphism.
    \item $H^*_{\sT}(QM_{rel})$ is a free module over $H^*_{\sT}(pt)$ and the classes $\overline{F_i^+}$ for all fixed components $F_i$ form a basis of $H^*_{\sT}(QM_{rel})$ over $H^*_{\sT}(pt)$.
    \item For any subgroup $\sS \subset \sT$, the natural map 
    \[ 
        H^*_{\sT}(QM_{rel})\otimes _{R(\sT)}R(\sS) \to H^*_{\sS}(QM_{rel})
    \]
    is an isomorphism.
\end{enumerate}
\end{proposition}

\section{Action of the Universal Enveloping Algebra of $\mathfrak{gl}_n$}\label{sec:action}

\subsection{Action on $H^*_\sT(QM_{ns})$}

First, we recall the construction in \cite{feigin2011gelfand} of the $U(\mathfrak{gl}_n)$ action on $H^*_\sT(QM_{ns})$. 

For our purpose, it's convenient to start from the homogenized enveloping algebra $U'(\mathfrak{gl}_n)$: it is an algebra over $\mathbb{C}[\epsilon ]$ with generators $E_i, F_i$ for $i=1,...,n-1$ and $H_i$ for $i=1,...,n$, satisfying the relations
\begin{equation}
    \label{eqn:gln_relations}
\begin{split}
    [E_i, F_i] &= \epsilon (H _{i+1}-H_i) \\
    [H_i, E_i] = -\epsilon E_i, &\;\; [H_{i+1}, E_i] = \epsilon E_i \\
    [H_i, F_i] = \epsilon F_i, &\;\; [H_{i+1}, F_i] =  -\epsilon F_i
\end{split}
\end{equation}
and the usual Serre relations. The $\epsilon $ here will correspond to the $\epsilon $ in equivariant variables. Specializing it to any non-zero complex number gives the usual $U(\mathfrak{gl}_n)$ (after dividing all generators by $\epsilon $).
These generators are realized geometrically as follows.

Let ${C}_{ns,i} ^{\bd }$ be the moduli space of flags of locally free sheaves
\begin{equation}\label{eqn:flag_for_corr}
    \mathcal{V}_1 \to \mathcal{V}_2 \to \cdots \mathcal{V}_{i}' \to \mathcal{V}_i \to \cdots  \to \mathcal{V}_n \simeq \mathcal{O}_{\mathbb{P}^1}^{\oplus n}
\end{equation}
such that 
\begin{itemize}
    \item $\text{rk}\;\mathcal{V}_k=k$ for $k=1,...,n-1$. $\text{rk}\;\mathcal{V}_i' = i$.
    \item Each map is an inclusion of sheaves
    \item $\mathcal{V}_i / \mathcal{V}_i' = \mathcal{O}_0$, where $\mathcal{O}_0$ denotes the skyscraper sheaf supported at $0 \in \mathbb{P}^1$.
\end{itemize}


It has two natural projections, to $QM_{ns}^{\bd } $ and $QM_{ns}^{\bd + \delta _i}$, denoted by $p$ and $q$. 
The action of $E_i, F_i$ are given by
\begin{equation}\label{eqn:ef_act}
    E_i = -q_*p^*, \;\; F_i = p_*q^*
\end{equation}
The action of Cartan elements $H_i, i=1,...,n$ are given by
\begin{equation}
    H_i = a_i + (d_{i-1} - d_i +i)\epsilon 
\end{equation}

\subsection{Action on $H^*_{\sT}(QM_{rel})$}

Now we turn to relative quasimaps.
To ease the notation, we will first consider $H^*_{\sT}(QM_{rel}):= H^*_{\sT}(QM_{rel, x_0})$ in this and next section. But the discussion also applies to $H^*_{\sT}(QM_{rel, w(x_0)})$ for any $w \in W$.

Note that the discussion in the previous section is for $QM_{ns,x_0}$. By symmetry, this construction also
applies to $QM_{ns, ev(\infty )=w(x_0)}$ for any $w \in W$, as long as we replace the $H_i$ action by 
\begin{equation}
    H_i = a_{w(i)} + (d_{i-1} - d_i +i)\epsilon 
\end{equation}

A more intrinsic way of writing this is to use the 
tautological bundles, namely
\begin{equation}
    H_i = c_1(\mathcal{F}_i) - c_1(\mathcal{F}_{i-1})+ i \epsilon 
\end{equation}
($c_1(\mathcal{F}_n) = c_1(\mathcal{F}_0) = 0$.)

The action on $H^*_{\sT}(QM_{rel})$ can be defined entirely analogous to $H^*_{\sT}(QM_{ns})$. For this, define the correspondences $C ^{\bd }_i$ in the same way as in ${C}_{ns,i} ^{\bd }$ using \eqref{eqn:flag_for_corr}, except that the bundles $\mathcal{V}_i$ are over (possibly) a chain of $\mathbb{P}^1$, with the stability condition as in the definition of $QM_{rel}$. As before, we define 
\begin{equation}\label{eqn:ef_act_rel}
    E_i = -q_*p^*, \;\; F_i = p_*q^*
\end{equation}
\begin{equation}\label{eqn:H_act}
    H_i = c_1(\mathcal{F}_i) - c_1(\mathcal{F}_{i-1})+ i \epsilon 
\end{equation}

\begin{theorem}
The $E_i, F_i, H_i$ above satisfy the relations in \eqref{eqn:gln_relations} and thus give rise 
to an action of $U(\mathfrak{gl}_n)$ on $H^*_{\sT}(QM_{rel})$.
\end{theorem}

\begin{proof}
As discussed in Section \ref{subsec:fixed_locus}, each fixed point in $QM_{rel}^{\bd }$ is labeled by $(\bla, \widetilde{P})$ where $\bla$ is a fixed point in $QM _{ns}^{\bd _1}$
and $\widetilde{P}$ is a fixed component in $\widetilde{QM}^{\bd _2}$ with $\bd _1 + \bd _2 = \bd $. 
The main result of \cite{feigin2011gelfand} shows that the relations are satisfied for fixed points
with $\bd _2 = 0$ (i.e. for $QM_{ns}$). To show this for other fixed components, note that by Proposition \ref{prop:tangent_add}, we have
\[ 
     E_i |\alpha  \rangle _{(\bla, \widetilde{f})} = \sum_{|\bmu| - |\bla|= \delta _i} c _{\bla \bmu} |\alpha  \rangle _{(\bmu, \widetilde{f})}
\]
where $c _{\bla \bmu}$ is equal to the coefficient of $|\bmu \rangle $ in $E_i|\bla \rangle $ in $QM_{ns, w(x_0)}^{\bd _1}$. 
(And similarly for $F_i$.) The weights of $\mathcal{F}_i$'s change with $ev(\infty' )$, and this makes the action of $H_i$ exactly match the action of 
$H_i$ on $QM _{ns, w(x_0)}^{\bd _1}$. 
Thus, the relations \eqref{eqn:gln_relations} are satisfied.

\end{proof}

\begin{remark}
Note that the $\mathcal{F}_i$'s are non-trivial vector bundles in general, so 
the action of the $H_i$'s may be non-semisimple after specializing equivariant parameters.
This already happens in the simplest example of $n=2$. We will come back to this point 
in Section \ref{subsec:catO}.
\end{remark}

\begin{remark}\label{rmk:uq_action}
\cite{braverman2005finite} considered the same correspondences 
on equivariant K-theory of $QM_{ns}$ and they form the quantum group $U_q(\mathfrak{gl}_n)$. To construct a $U_q(\mathfrak{gl}_n)$
action on $K_\sT(QM_{rel})$, one needs to be careful to avoid square roots of equivariant variables, as the line bundles $\det(\mathcal{F}_i)$ may not have a square root.
One way to do this is that, instead of defining the action of the 
Drinfeld-Jimbo generators $e_i, f_i, i=1,...,n-1$ and  $ \psi _j, j=1,..n$ (which will involve square roots), one defines the action of $E_i = \psi _{i+1}e_i, F_i = f_i \psi _i$ and $\Phi _j = \psi _j^2$.
\end{remark}

\begin{remark}\label{rmk:bigger_alg}
A bigger algebra action (Yangian $Y(\mathfrak{gl}_n)$ in the case of cohomology and quantum affine algebra $U_q(\widehat{\mathfrak{gl}}_n)$ in the case of K-theory) can be constructed as in \cite{feigin2011yangians}, \cite{tsymbaliuk2010quantum}, \cite{neguct2018affine} by further twisting the correspondences by the tautological line bundle. It's not hard to see that these actions 
also extends to $H^*_{\sT}(QM_{rel})$ and $K_{\sT}(QM_{rel})$ by analyzing the weights at fixed locus similar to the proof above.
\end{remark}

\section{Specializing to Regular Lowest Weight}
\label{sec:spec_lowest_wt}
Fix $\lambda = (\lambda_1,...,\lambda_n)$ where $\lambda_1 > \lambda_2 > ... > \lambda_n$ are integers.
(We can get other permutations of $\lambda$ by letting $ev(\infty )=w(x_0)$ for different $w \in W$. So we fix the $\lambda _i$'s to be in decreasing order.)
Fix a generic complex number $\epsilon _0$ and fix $w \in W$.
Consider the $U(\mathfrak{gl}_n)$ representation
\[ 
    H^*_{\sT}(QM_{ns, w(x_0)})\otimes _{R(\sT)} \mathbb{C}_\lambda
\]
where $\mathbb{C}_\lambda$ becomes a module over $R(\sT)$ by 
\begin{equation}\label{eqn:spec_wt}
    a_i \mapsto \lambda_i \epsilon, \epsilon \mapsto \epsilon _0
\end{equation}

\begin{theorem} (\cite{feigin2011gelfand}, Theorem 3.5)
The module $H^*_{\sT}(QM_{ns, w(x_0)})\otimes _{R(\sT)} \mathbb{C}_\lambda$ is isomorphic     
to the dual Verma module of $U(\mathfrak{gl}_n)$ with lowest weight $w(\lambda)-\rho $.
\end{theorem}

It's natural to ask what we get if we do the same thing for $QM_{rel, w(x_0)}$. 
Instead of doing this for the whole $H^*_{\sT}(QM_{rel, w(x_0)})$, we will first pick out a submodule
depending on the lowest weight we are considering.

In the rest of this section, we omit the $w(x_0)$ and simply write $QM_{rel}$ for brevity.
As vector spaces, 
\[ 
    H^*_{\sT}(QM_{rel})\otimes _{R(\sT)} \mathbb{C}_\lambda \simeq H^*\left(QM_{rel}^{\mathbb{C}^*_\lambda}\right)
\]
where we use $\mathbb{C}^*_\lambda$ to denote the subtorus of $\sT$ whose Lie algebra spans the subspace 
\[ 
    a_i = \lambda_i \epsilon 
\]
in $\text{Lie}(\sT)$. Furthermore,
the correspondences $C_i ^{\bd }$ can also be replaced by their $\mathbb{C}^*_\lambda$ fixed points
with suitable twists by normal bundles. (Cf. \cite{chriss1997representation} Section 5.5 and 5.11)


For each $\bd $, the connected components of $(QM_{rel}^{\bd })^{\mathbb{C}^*_\lambda}$ can be divided into two groups: the components that intersect $QM _{ns}$ and the components that do not.
Let 
\[ 
    I_0 := \{L \in  \text{connected components of }(QM_{rel}^{ })^{\mathbb{C}^*_\lambda}| L \cap QM_{ns} \neq \emptyset\}
\]
and $I_1$ be its complement.
Then
\begin{equation}
    \label{eqn:Hp}
    H_\lambda := \bigoplus_{L \in I_0} H^*(L) \text{  and  } H_\lambda' := \bigoplus_{L \in I_1} H^*(L)
\end{equation} 
each form a submodule of $(QM_{rel}^{\bd })^{\mathbb{C}^*_\lambda}$,
since the correspondences preserve $QM_{ns}$.

Let $\CM_0 := \coprod _{L \in I_0} L$.

\begin{lemma}\label{lem:fp_in_closure}
A $\sT$ fixed point $(\bla, \widetilde{f})$ is contained in $\CM_0$ if and only if the $\mathbb{P}^1$'s in the domain of $\widetilde{f}$ each covers an invariant $\mathbb{P}^1$ in the target, and the weight of each $\mathbb{P}^1$ (see Remark \ref{rmk:p1weight}) is equal to $\epsilon $ under the specialization \eqref{eqn:spec_wt}. 
\end{lemma}

\begin{proof}
The fixed locus $(QM_{rel}^{\bd })^{\mathbb{C}^*_\lambda}$ is smooth and irreducible. \footnote{Note that we are considering the action of a $\mathbb{C}^*$ here. This may fail if we consider a \textit{finite} group acting on a Deligne-Mumford stack.} So $I_0$ is equal to the closure of $(QM_{rel}^{\bd })^{\mathbb{C}^*_\lambda} \cap QM_{ns}$ inside $QM_{rel}$. So we need to decide which fixed points $(\bla, \widetilde{f})$ can be deformed to a map in $QM_{ns}^{\mathbb{C}^*_\lambda}$. This is equivalent to the weight of smoothing each node (which is the sum of tangent weights of the two $\mathbb{P}^1$'s next to the node) is trivial under the specialization \eqref{eqn:spec_wt}. The tangent weight of the parametrized $\mathbb{P}^1$ at the first node is $-\epsilon $, so all $\mathbb{P}^1$'s in the bubble need to have weight $\epsilon $.
\end{proof}

\section{Structure of $H_{\lambda,w}$}\label{sec:Hpw}
From now on, we will focus on the study of the module $H_\lambda$ defined in \eqref{eqn:Hp}. Denote by $H_\lambda ^{\bd }$ the subspace of it coming from degree $\bd $ quasimaps.
Later in this section, when we want to stress that the evaluation point is $w(x_0)$, we will write $H_{\lambda,w}$ for this module.



\subsection{The Category $\mathcal{O}'$}
\label{subsec:catO}
As mentioned earlier, the Cartan elements may act non-semisimply on $H_\lambda$. However, one can prove that elements in the center of $U(\mathfrak{gl}_n)$ act semisimply (and are thus multiplication by constants). Let $Z(U(\mathfrak{gl}_n))$ be the center.

\begin{lemma}
For any element $z \in Z(U(\mathfrak{gl}_n))$, the action of $z$ on $H_\lambda$ is multiplication by a constant.
\end{lemma}
\begin{proof}
For $H^*_{\sT}(QM_{ns, w(x_0)})$, the action of $z$ is multiplication by a polynomial $m_z(a_1,...,a_n, \epsilon )$.
$m_z$ is symmetric in $a_i$'s and hence does not depend on the choice of $w$. For each $\sT$-fixed component $(\bla, \widetilde{P})$ in $H^*_{\sT}(QM_{rel})$, the action of $z$ only depends on $\bla$, so it's also multiplication by $m_z$.
Thus, after specializing $a_i$ and $\epsilon $, the action of $z$ on $H_\lambda$ is a constant.
\end{proof}

This is opposite to the usual category $\mathcal{O}$ where 
$H_i$'s act semisimply while the center typically acts non-semisimply. Define the category $\mathcal{O}'$ to be 
the category with the above properties.
Namely, a module $M$ is in the category $\mathcal{O}'$ if
\begin{enumerate}[label=(\roman*)]
    \item $M$ is finitely generated
    \item $M$ is locally $\mathfrak{n}$-finite where $\mathfrak{n}$ is the span of $H_i$ and $F_i$ for $i=1,...,n$. 
    \item The center $Z(U(\mathfrak{gl}_n))$ acts on $M$ by constants.
\end{enumerate}
The result in \cite{soergel1986equivalences} shows that 
there is an equivalence of categories 
\[ 
    \Upsilon: \mathcal{O} \xrightarrow[]{\sim} \mathcal{O}'
\]
We need to determine where each module maps to under $\Upsilon$. (Cf. \cite{milicic1997composition} Section 5.)
\begin{proposition}
Let $\lambda $ be a dominant weight. For any $w \in W$,
the functor $\Upsilon$ sends simple (resp. Verma, dual Verma) module  of lowest weight $w \cdot \lambda $ to 
simple (resp. Verma, dual Verma) module  of lowest weight $w ^{-1} \cdot \lambda $
\end{proposition}

\begin{proof}
The statement about simple module follows from \cite{jantzen2013einhullende} Proposition 6.34.
The Verma module is the projective cover of the simple module in a truncated category as in \cite{beilinson1996koszul} (and the truncated categories on the two sides match). So the statement about Verma module follows. Same for dual Verma modules.
\end{proof}

\subsection{Tilting module}
\hide{Maybe we change $H_\lambda$ to a curly font.}
We first show that $H_\lambda$ has a filtration by dual Verma modules. This comes from the stratification on the space $\CM_0$.

Take the torus for constructing the Bialynicki-Birula decomposition to be
\[ 
    a_1 \gg a_2 \gg ... \gg a_n \gg \epsilon. 
\]
(This means that e.g. weight $a_1 - a_2$ is attracting.)
Recall that $\CM_0$ is a subset of $\mathbb{C}^*_\lambda$ fixed locus of $QM_{rel}$, so the torus $\sT$ acts on $\CM_0$ and the above torus induces a B-B decomposition.

For a $\sT$-fixed point $(\bla, \widetilde{f})$ in $\CM_0$, use 
$\zeta _{\bla, \widetilde{f}}$ to denote the closure of the attracting locus of $(\bla, \widetilde{f})$. We will also use it to denote its class in $A^*_\sT(\CM_0)$.

The above B-B decomposition can be described as follows: For each 
fixed point $w(x_0)$ in the flag variety, let $X_w$ denote the Schubert cell induced by the torus above. 
The attracting locus of
$(\bla, \widetilde{f})$ can be written as $\zeta _\bla \times N \times \zeta _{\widetilde{f}}$, where 
\begin{itemize}
    \item $\zeta _\bla$ denotes the B-B cell in $QM _{ns}^{\mathbb{C}^*_\lambda}$. This corresponds to deforming the map on the first $\mathbb{P}^1$
    \item $\zeta _{\widetilde{f}}$ denotes the B-B cell in $\widetilde{QM}^{\mathbb{C}^*_\lambda}$ with fixed $ev_0$. This corresponds to deforming the map on the bubbles.
    \item $N$ is a unipotent subgroup of $GL(n)$ such that the action of $N$ on $w(x_0)$ gives an isomorphism from $N$ to $X_w$. This corresponds to moving the first node by multiplying by $N$.
\end{itemize}

Based on this, we can choose a partial ordering on the fixed points as follows: 
We say that $(\bla, \widetilde{f}_1)> (\bmu, \widetilde{f}_2)$ if
\begin{itemize}
    \item $ev_0(\widetilde{f}_1)> ev_0(\widetilde{f}_2)$ in the Schubert cell order (or equivalently, the Bruhat order on $w \in W$.) 
    \item $ev_0(\widetilde{f}_1)= ev_0(\widetilde{f}_2)$, and $\widetilde{f}_1 > \widetilde{f}_2$ in the B-B decomposition on $\widetilde{QM}^{\mathbb{C}^*_\lambda}$
    \item Both of the first two comparisons are equal, and $\bla> \bmu$ in the BB-decomposition of $QM_{ns, ev(\infty )=ev_0(\widetilde{f})}$.
\end{itemize}

Now consider the sequence of subsets
\[ 
  \emptyset = U_0 \subset   U_1 \subset U_2 \subset ... \subset U_m = \CM_0
\]
such that each $U _{i+1} \backslash U_i$ is 
\[ 
    \bigcup _{\bla} \text{attracting set of} (\bla, \widetilde{f})
\]
for a given $\widetilde{f}$ and such that these $\widetilde{f}$ appears in the sequence of the partial order above.
Then each $U_i$ is an open subset of $\CM_0$. This gives a sequence of surjections 
\begin{equation}
    \label{eqn:dual_verma_filt}
    0 \twoheadleftarrow  H^*(U_1) \twoheadleftarrow H^*(U_2) \twoheadleftarrow ... \twoheadleftarrow H^*(U_m) = H^*(\CM_0)
\end{equation}
By construction, $\ker (U _{i+1} \twoheadrightarrow U_i)$ is isomorphic to a dual Verma module for each $i$. 
\hide{To see this more precisely, we need to argue that the inclusion of $QM_{ns} \times N$ is a closed embedding here, 
and the normal bundle is trivial because we are swiping a trivial bundle using $N$.}
This implies a dual Verma filtration by taking $\ker (H^*(U_i)\twoheadleftarrow H^*(U_m)) $ as the $i$-th term.

Let $\varpi$ be the map from $\CM_0$ to a point. 
Poincare duality for orbifolds \cite{adem2007orbifolds} implies that the pairing 
\[ 
    (\alpha , \beta ) \mapsto \varpi_*(\alpha \cup \beta )
\]
is non-degenerate. This pairing is compatible with the $U(\mathfrak{gl}_n)$ action since
\[ 
    (\alpha , F_i \beta ) = (p^* \alpha , q^* \beta ) = (-E_i \alpha , \beta ).
\]
So one can dualize the sequence \eqref{eqn:dual_verma_filt} to get 
\[ 
    0 \hookrightarrow M_1 \hookrightarrow M_2 ... \hookrightarrow M_m = H_\lambda
\]
such that each successive quotient $M_{i+1}/M_i$ is a Verma module.

Now we know that $H_\lambda$ has both Verma and dual Verma filtrations.
Under the categorical equivalence, the same holds for $\Upsilon ^{-1}(H_\lambda)$.
So $\Upsilon ^{-1}(H_\lambda)$ is a tilting module.
In other words, $H_\lambda$ is the image of a tilting module under $\Upsilon$.

\subsection{Multiplicities}
In this section, the point $ev(\infty )$ will become important, so we restore it in our notation. We will write $H_{\lambda, w}$
for the module we get from $QM_{rel, w(x_0)}$. We have shown that each $H_{\lambda, w}$ is the image of a tilting module under $\Upsilon$. 

Every tilting module is a direct sum of indecomposible tilting modules, and the indecomposible ones are parametrized by the lowest weight. 
Let $T(\lambda)$ denote the indecomposible tilting module of lowest weight $\lambda$ and $V(\lambda)$ denote 
the Verma module of lowest weight $\lambda$.

The dimension of degree $\bd $ weight space in $H_{\lambda,w}$
is equal to the number of $\sT$-fixed points in $\CM_0^\bd $. This can be used to determine the multiplicities.
Let $v _{\bd }$ denote the number of $\sT$-fixed points in $QM_{ns}^{\bd }$. (Set $v _\bd =0$ if not all entries of $\bd $ are non-negative.) 
This is equal to the dimension of degree $\bd $ weight space in a Verma module.
For any $u,w \in W$,
let $b _{w, u}$ be the number of paths from $w$ to $u$ in the Bruhat graph. (The arrow points to longer elements in the Bruhat graph.)
Let $\bd (u-w)$ be the vector $(d_1,...,d_{n-1})$ such that 
\[ 
    u(\lambda) - w(\lambda) = \sum_{i = 1}^{n-1}d_i \alpha _i
\]
where $\alpha _i, i=1,...,n-1$ are the simple positive roots.

\begin{lemma}
The dimension of $H^*(\CM_{w(x_0)} ^{\bd })$ is equal to 
\[ 
    \sum_{u \in W} b _{w,u} v _{\bd -\bd (u-w)}
\]
\end{lemma}
\begin{proof}
Given a fixed point $f$, suppose that the evaluation of the first node is $u(x_0)$, then the possible maps on bubbles is in bijection with Bruhat paths from $u$ to $w$, 
and the bubbles occupies degree $\bd (u-w)$.
So the number of possible maps on the parametrized $\mathbb{P}^1$ is $ v _{\bd -\bd (u-w)}$, thus the conclusion.
\end{proof}

\begin{corollary}
The multiplicity of $V(u(\lambda)-\rho )$ in $\Upsilon ^{-1}(H_{\lambda, w})$ is equal to $b_{w,u ^{-1}}$ for $w \preceq u ^{-1}$ in the Bruhat order. Verma modules of other lowest weights do not appear in $H_{\lambda, w}$.        
\end{corollary}

Let $p _{u,w}$ be the Kazhdan-Lusztig polynomial $P _{u,w}$ evaluated at 1. The
multiplicities of Verma modules in tilting modules, see e.g. \cite{humphreys2021representations}, can be expressed as
\[ 
   (T(y(\lambda)-\rho ): M(u(\lambda)-\rho ))   = p _{uw_\circ , y w_\circ }
\]
where $w_\circ$ is the longest element in $W$.
So we have 
\begin{corollary}
\[ 
    \Upsilon ^{-1}(H_{\lambda, w}) = \bigoplus _{y \in W} T(y(\lambda)-\rho )^{\oplus n _{w,y}}
\]
where $n _{w,y}$ is determined by the relation
\[ 
    \sum_{y}n _{w,y} p _{u w _{\circ } , y w _{\circ }} = b _{w, u ^{-1}}
\]
Equivalently,
\[ 
    n _{w,y} = \sum_{u} b _{w, u ^{-1}}p _{u,y} (-1)^{l(u)-l(y)}
\]
\end{corollary}

\todo{Singular vectors can be identified with the $\widetilde{QM}$ part by counting dimensions}

\todo{The tangent line bundle $\psi $ is an endomorphism of $H_{\lambda,w}$. $\psi -\epsilon $ is nilpotent.}

\todo{SL2 example}

\end{document}